\documentclass{article}
\usepackage{amsmath,amsthm,amsfonts}

\allowdisplaybreaks[1]
\newtheorem{thm}{Theorem}[section]
\newtheorem{lem}[thm]{Lemma}
\newtheorem*{cor}{Corollary}

\theoremstyle{definition}

\theoremstyle{remark}
\newtheorem{rem}{Remark}

\newcommand{\ud}{\mathrm{d}}
\numberwithin{equation}{section}  % If you number theorems, etc. within sections,
\begin{document}%\recd{}{}%Do not alter this line.

\title{A Confirmation of a Conjecture on the Feldman's\\ Two-armed Bandit Problem} % insert title - use \\ if it requires more than one line.
%\authorone[Shandong University]{Zengjing Chen} 
%\authorone[Shandong University]{Yiwei Lin}
%\authoronethree[Shandong University]{Jichen Zhang}
%\addressone{School of Mathematics, Shandong University, Jinan 250100, China.}% Your postal address goes here.
%\emailthree{jichenzhang@mail.sdu.edu.cn} %Authors email goes here.
%\author[Z. Chen]{Zengjing Chen}
%\author[Y. Lin]{Yiwei Lin}
\author{Zengjing Chen, Yiwei Lin, Jichen Zhang\thanks{Corresponding author: jichenzhang@mail.sdu.edu.cn}}
%\address{School of Mathematics, Shandong University, Jinan 250100, China.}
%\email{jichenzhang@mail.sdu.edu.cn}
%\subjclass[2020]{62C10;62L05}

%\keywords{Myopic strategy; stochastically maximizing; dynamic programming property}%insert keywords separated by a semicolon. You should avoid including keywords which also appear in the title.

%\ams{62C10}{62L05}% insert the primary 2020 Maths Subject Classification number in the first bracket
		% and the secondary ams number(s) in the second bracket
		% e.g. \ams{60E20}{49G03;49F10}
		%Maximum of three in each, ideally one or two in each primary and secondary.
		%codes found here ``https://mathscinet.ams.org/msnhtml/msc2020.pdf''
\maketitle

\begin{abstract}
Myopic strategy is one of the most important strategies when studying bandit problems. In this paper, we consider the two-armed bandit problem proposed by Feldman. With general distributions and utility functions, we obtain a necessary and sufficient condition for the optimality of the myopic strategy. As an application, we could solve Nouiehed and Ross's conjecture for Bernoulli two-armed bandit problems that myopic strategy stochastically maximizes the number of wins.
\end{abstract}

%%%%%%%%%%%%%%%%%%%%%%%%%%%%%%%%%%%%%%%%%%%%%%
%%%% Main text entry area:
\section{Introduction}%
\label{sec:introduction}

The bandit problem is a well-known problem in sequential control under conditions of incomplete information. It involves sequential selections from several options referred to as arms of the bandit. The payoffs of these arms are characterized by parameters which are typically unknown. Agents should learn from the past information when deciding which arm to select next with an aim to maximize the total payoffs.

This problem can be traced back to Thompson's work \cite{Thompson1933} related to medical trials. Now it is widely studied and frequently applied as a theoretical framework for many other sequential statistical decision problems in market pricing, medical research, and engineering, which characterized by the trade-off between exploration and exploitation (see e.g. \cite{Garbe1998,Kim2016,Rusmevichientong2010}).

Here we focus on the two-armed bandit problem which are studied by Feldman \cite{Feldman1962}. For a given pair $(F_1,F_2)$ of distributions on a probability space $(\Omega,\mathcal{F},P)$, consider two experiments $X$ and $Y$ (called $X$-arm and $Y$-arm), having distributions under two hypotheses $H_1$ and $H_2$ as follows:
\begin{equation}\label{bandit}
	\begin{array}{rccc}
		&~&~X&~Y\\
		(\xi_0)&~H_1:&~F_1&~F_2\\
		(1-\xi_0)&~H_2:&~F_2&~F_1,\\
	\end{array}
\end{equation}
where $\xi_0$ is the priori probability that $H_1$ is true.
In trial $i$, either $X$-arm or $Y$-arm is selected to generate a random variable $X_i$ or $Y_i$ which describes the payoff, and $\xi_i$ is the posterior probability of $H_1$ being true after $i$ trials. The aim is to find the optimal strategy that maximizes the total expected payoffs.

Among the many notable strategies such as myopic strategy, Gittins strategy, and play-the-winner strategy, \textit{myopic strategy} is undoubtedly one of the most appealing strategies. With this strategy, in each trial, agents select the arm with greater immediate expected payoff, i.e., play each time as though there were but one trial remaining. Mathematically, let  $E_P[ \cdot\vert H_1]$ be the expectation functional under hypothesis $H_1,$ if
\begin{equation}\label{eq:classic_condition}
	E_P[X_1\vert H_1]\geq E_P[Y_1\vert H_1],
\end{equation}
which is equivalent to
\begin{equation}
	\int_{\mathbb R} xdF_1(x)\ge \int_{\mathbb R} xdF_2(x),
\end{equation}
then agents select $X$-arm in trial $i$ when $\xi_{i-1}\ge \frac{1}{2}$, or $Y$-arm otherwise.

When the myopic strategy is optimal, it means that the optimal strategy is time invariant, i.e., it does not depend on the number of trials remaining. Hence the optimal strategy can be easily implemented.
Unfortunately, myopic strategy is not optimal in general. This is mainly because at each time myopic strategy only considers the payoff of the next trial; however, to maximize the total payoffs, all the remaining trials should be considered. Kelley \cite{Kelley1974}, Berry and Fristedt \cite{Berry1985} showed  counterexamples that myopic strategy is longer optimal. It is an open question to know under what conditions the myopic strategy is optimal. The optimality of myopic strategy has always attracted people's attention. Thompson \cite{Thompson1933} is the first who
gave a strategy which estimates the probability $P_{\{X\ge Y\}}$ that arm $X$ is better than $Y$ at each stage, and allocates the next trial according to $P_{\{X\ge Y\}}$. This gives the first idea of myopic strategy.Later, Bradt et al. \cite{Bradt1956} considered  Model \eqref{bandit} when $F_1$ and $F_2$ in Model \eqref{bandit} are Bernoulli with expectation $\alpha$ and $\beta$ respectively, they showed that myopic strategy is optimal when $\alpha+\beta=1$. They conjectured that myopic strategy is also optimal when $\alpha+\beta\neq 1$, and verified this conjecture for $n\le 8$. Feldman \cite{Feldman1962} showed a sufficient condition \eqref{eq:classic_condition} for the optimality of myopic strategy in arbitrary number of trials. Kelley \cite{Kelley1974} showed some  necessary conditions  for a Bernoulli two-armed bandit model. Rodman \cite{Rodman1978} extended Feldman's result to a multi-armed setting.

The exploration of myopic strategy under what conditions is the optimal strategy has not come to an end. This problem remains open under more general settings. Nouiehed and Ross \cite{Nouiehed2018} studied the Bernoulli armed bandit problem and posed a conjecture that myopic strategy also maximizes the probability that not less than $k$ wins occur in the first $n$ trials, for all $k,n$. They proved this conjecture for $k=1$ and $k=n$ in $m$-armed bandit, and for $k=n-1$ in two-armed bandit problem.

Why hasn't the question in Nouiehed and Ross's conjecture been raised for almost sixty years? Nouiehed and Ross \cite{Nouiehed2018} explained that this was because in Feldman \cite{Feldman1962} and similar studies, it was the number of times the better arm was chosen that was maximized, not the total payoff. Although the two approaches are equivalent in \cite{Feldman1962}, they are quite different when we want to study a more general utility function.

This opens up a whole new horizon for us to study this issue. All works mentioned above considered the utility function $\varphi (x)=x$ (e.g. \cite{Berry1985,Feldman1962,Kelley1974}) or $\varphi(x)=I_{[k,+\infty)}(x)$ (e.g. \cite{Nouiehed2018}), so a natural question is what conditions can guarantee the optimality of myopic strategy for general utility functions.

In this paper, we focus on the optimal strategy for the most typical case of two-armed bandit problems proposed in the profound paper of Feldman \cite{Feldman1962}. With a general utility function to be considered, we obtain a necessary and sufficient condition for the optimality of myopic strategy. As an application, we could solve Nouiehed and Ross's conjecture for two-armed case.

We consider a situation that the agent playing bandit \eqref{bandit} has a utility function $\varphi$ and starts with an initial fund of $x$ and a strategy $\mathsf{M}^n$ : in trial $i$, play $X$-arm if $\xi_{i-1}\ge \frac{1}{2}$, or $Y$-arm otherwise. The innovative aspects of the obtained results in this paper are as follows: firstly, we take $F_1$ and $F_2$ as general distribution functions, continuous or not, rather than Bernoulli distributions; secondly, we consider general utility functions which are no longer linear. This makes Feldman's proof method invalid and brings some additional difficulties. We shall show that $\mathsf{M}^n$ maximizes the expected utility
 of $n$ trials if and only if the utility function $\varphi$ and the distributions $F_1$ and $F_2$ satisfy
\begin{align}\label{eq-1}
	E_P[\varphi(u+X_1)\vert H_1]\ge E_P[\varphi(u+Y_1)\vert H_1], \text{ for any $u\in\mathbb{R}$.}
\end{align}
%where $E_P[\cdot|H_1]$ is the expectation functional under hypothesis $H_1$, i.e.,
%\begin{equation}
%	\int\varphi(u+t)dF_1(t)\ge \int\varphi(u+t)dF_2(t), \text{ for any $u\in\mathbb{R}$.}
%\end{equation}
Condition \eqref{eq-1} means that no matter how much money the agent already has, if only one trial is to be played, playing the arm with distribution $F_1$ is always better than playing the arm with $F_2$. In the case that $\varphi(x)=x$, Condition \eqref{eq-1} coincides with Condition \eqref{eq:classic_condition}.

It is interesting that if we choose the utility function in Condition \eqref{eq-1}  as an indictor function $\varphi(x)=I_{[k,+\infty)}(x)$, and initial fund $u=0$, we could prove the Nouiehed and Ross's conjecture for two-armed case immediately.

The structure of the paper is as follows. In Section \ref{sec:notations}, we describe the two-armed bandit problem and some basic properties. In Section \ref{sec:main_results}, we first introduce a dynamic programming property of the optimal expected utility, and then we prove the main result, as a corollary, we derive the validity of  Nouiehed and Ross's conjecture in the two-armed bandit case.

\section{Preliminaries}%
\label{sec:notations}

Let us start with the description of the two-armed bandit model and the strategies.

Consider the bandit model in Equation \eqref{bandit}. Let $\{X_i\}_{i\geq1}$ be a sequence of random variables where $X_i$ describes the payoff of trial $i$ from $X$-arm, and $\{Y_i\}_{i\geq1}$ be a sequence of random variables selected from $Y$-arm. $\{(X_i,Y_i)\}_{i\ge 1}$ are independent under each hypothesis. We define $\mathcal{F}_i:=\sigma\{(X_1,Y_1),\cdots,(X_{i},Y_i)\}$, which represents all information obtained until trial $i$.

\emph{We call this model a $(\xi_0,n,x)$-bandit}, if there are $n$ trials to be played  with initial fund $x$ and a prior probability $\xi_0$. In our following discussion, the distributions $F_1$ and $F_2$ of arms are continuous with density $f_1$ and $f_2$, respectively.

\begin{rem}
	The same results still hold, when the distributions of arms are discrete, e.g. Bernoulli. We only need to modify the calculation of expectations in this case.
\end{rem}

For each $i\geq1$, let $\theta_i$ be an $\mathcal{F}_{i-1}$-measurable random variable taking values in $\{0,1\}$, where $\theta_i=1$ means $X$-arm is selected for observation in trial $i$ and $\theta_i=0$ means $Y$-arm is selected for observation in trial $i$.
Let $\Theta_n$ be the set of strategies $\theta=\{\theta_1,\cdots,\theta_n\}$ for $n$-trial bandits. The payoff that an agent receives with strategy $\theta$ in trial $i$ is
\begin{align}
	Z_{i}^{\theta}:=\theta_{i}X_{i}+(1-\theta_{i})Y_{i},~1\leq i\leq n.
\end{align}

For a $(\xi_0,n,x)$-bandit and a suitable measurable function $\varphi$, the \emph{expected utility} obtained by using strategy $\theta$ is denoted by
\begin{align}\label{eq:w_defn_old}
	W(\xi_0,n,x,\theta)=E_P\biggl[\varphi\biggl(x+\sum_{i=1}^{n} Z_i^\theta\biggr)\biggr],
\end{align}
where $\varphi$ is called a utility function.

For each strategy $\theta\in\Theta_n$, let $\{\xi_i^\theta\}_{i\ge1}$ be  the sequence of the posterior probabilities that hypothesis $H_1$ is true after $i$ trials. The posterior probability $\xi^\theta_1$ after trial $1$ with a payoff $s$ is calculated by
\begin{equation}\label{eq-3}
	\xi_1^\theta(s)=\left\{\begin{aligned}
		&\frac{\xi_0f_1(s) }{\xi_0 f_1(s)+(1-\xi_0)f_2(s)},&&\text{if $\theta_1=1$, i.e. $X$-arm is selected,}\\
		&\frac{\xi_0f_2(s) }{\xi_0 f_2(s)+(1-\xi_0)f_1(s)},&&\text{if $\theta_1=0$, i.e. $Y$-arm is selected.}
	\end{aligned}\right.
\end{equation}
We can easily obtain that for any fixed $s$, $\xi_1^\theta(s)$ is increasing in $\xi_0$.
When the posterior probability $\xi_i^\theta$ is known and the payoff of the  $i+1$ trial is $s$, there is a recursive formula 
\begin{equation}\label{posterior}
	\xi_{i+1}^\theta(s)=\left\{\begin{aligned}
		&\frac{\xi_i^\theta f_1(s) }{\xi_i^\theta  f_1(s)+(1-\xi_i^\theta)f_2(s)},&&\text{if $\theta_{i+1}=1$, i.e. $X$-arm is selected,}\\
		&\frac{\xi_i^\theta f_2(s) }{\xi_i^\theta f_2(s)+(1-\xi_i^\theta)f_1(s)},&&\text{if $\theta_{i+1}=0$, i.e. $Y$-arm is selected.}
	\end{aligned}\right.
\end{equation}

Now, we propose the following two-armed bandit problem:

\noindent{\bf Problem (TAB).} For a $(\xi_0,n,x)$-bandit and a utility function $\varphi$, find some strategy in $\Theta_n$ to achieve the maximal expected utility
\begin{equation}\label{eq-2}
	V(\xi_0,n,x):=\sup_{\theta\in \Theta_n} W(\xi_0,n,x,\theta)=\sup_{\theta\in\Theta_n} E_P\biggl[\varphi\biggl(x+\sum_{i=1}^{n} Z_i^\theta\biggr)\biggr].
\end{equation}

Note that the expected utility $E_P[\cdot]$ depends on hypothesis $H_1$ , $H_2$ and $\xi_0$. In fact,
\begin{equation}
\begin{aligned}
	&E_{P} \biggl[\varphi\biggl(x+\sum_{i=1}^{n} Z_i^\theta\biggr)\biggr]\\
	=&\xi_0 E_P\biggl[\varphi\biggl(x+\sum_{i=1}^{n} Z_i^\theta\biggr)\vert H_1\biggr]+(1-\xi_0)E_P\biggl[\varphi\biggl(x+\sum_{i=1}^{n} Z_i^\theta\biggr)\vert H_2\biggr],
\end{aligned}
\end{equation}
where $E_P[\cdot\vert H_i]$ is the expectation under hypothesis $H_i$ $(i=1,2).$

To simplify the notation, we write $E_P[\cdot\vert H_1]$ shortly as $E_1[\cdot]$, $E_P[\cdot\vert H_2]$ as $E_2[\cdot]$ and
$E_P[\cdot]$ as $E_{\xi_0}[\cdot]$. Then the expected utility can be written as
\begin{align}\label{eq:W_defn}
	W(\xi_0,n,x,\theta)=E_{\xi_0}\biggl[\varphi\biggl(x+\sum_{i=1}^{n} Z_i^\theta\biggr)\biggr],
\end{align}
where
\begin{equation}\label{E_xi}
		E_{\xi_0} \biggl[\varphi\biggl(x+\sum_{i=1}^{n} Z_i^\theta\biggr)\biggr]
		=\xi_0 E_1\biggl[\varphi\biggl(x+\sum_{i=1}^{n} Z_i^\theta\biggr)\biggr]+(1-\xi_0)E_2\biggl[\varphi\biggl(x+\sum_{i=1}^{n} Z_i^\theta\biggr)\biggr].
\end{equation}
Immediately, equality \eqref{eq-2} can be written as follows:
\begin{align}\label{eq:V_defn}
	V(\xi_0,n,x)=\sup_{\theta\in \Theta_n} W(\xi_0,n,x,\theta)=\sup_{\theta\in \Theta_n} E_{\xi_0}\biggl[\varphi\biggl(x+\sum_{i=1}^{n} Z_i^\theta\biggr)\biggr].
\end{align}

Consider a strategy  $\mathsf{M}^n$ : in trial $i$, play $X$-arm if $\xi_{i-1}\ge \frac{1}{2}$, or $Y$-arm otherwise. Our main result is to find conditions under which  $\mathsf{M}^n$ could solve Problem (TAB).

The following lemma shows that when calculating the expected utility, we can temporally fix the payoff of the first trial, calculate the expected utility as a function of the first payoff, and then take expectation while seeing the first payoff as a random variable. This is an extension of equation (2) in Feldman \cite{Feldman1962}.

\begin{lem}\label{lem:Exp}
	For each integer $n\ge 2$ and strategy $\theta=\{\theta_1,\cdots,\theta_n\}\in\Theta_n $, we have
	\begin{align}\label{eq:Exp_eq1}
		E_{\xi_0} \biggl[\varphi\biggl(x+\sum_{i=1}^{n} Z_i^\theta\biggr)\biggr]
		=E_{\xi_0}\left[h\bigl(x,Z_1^\theta\bigr)\right], \forall x\in \mathbb R,
	\end{align}
	where $h(x,u)=E_{\xi_1^\theta(u)}\left[\varphi\Bigl(x+u+\sum_{i=2}^{n} Z_i^{\theta[u]}\Bigr)\right],$
	$\xi_{1}^\theta$ is defined by \eqref{eq-3} and \mbox{$\theta[u]$} is the strategy obtained from $\theta$ by fixing the payoff of the first trial to be $u$.
\end{lem}
\begin{rem}
	$E_{\xi_1^\theta(u)}[\cdot]$ is the expected utility $E_{\xi_0}[\cdot]$ replacing $\xi_0$ with $\xi_1^\theta(u)$. In integral form, Equation \eqref{eq:Exp_eq1}  is
	\begin{equation}
	\begin{aligned}
		&E_{\xi_0} \biggl[\varphi\biggl(x+\sum_{i=1}^{n} Z_i^\theta\biggr)\biggr]\\
		=&\int_{\mathbb R} E_{\xi_1^\theta(u)}\biggl[\varphi\biggl(x+u+\sum_{i=2}^{n} Z_i^{\theta[u]}\biggr)\biggr]\left(\xi_0f_{1}(u)+(1-\xi_0)f_{2}(u)\right)\ud u.
	\end{aligned}
	\end{equation}
\end{rem}
\begin{rem}
	Here $h\left(x,Z^\theta_1\right)$ can be seen as a conditional expectation of $\varphi\!\left(x+\sum_{i=1}^{n} Z_i^\theta\right)$ given $\mathcal{F}_1$, and this lemma shows that it has the same expectation as $\varphi\left(x+\sum_{i=1}^{n} Z_i^\theta\right)$.
\end{rem}
\begin{proof}
	Fix some strategy $\theta=\{\theta_1,\cdots,\theta_n\}\in\Theta_n.$ In order to prove \eqref{eq:Exp_eq1}, without loss of generality, we assume that $\theta_1=1$, then $Z_1^{\theta}=X_1.$ Note that $\theta_i$ is $\mathcal{F}_{i-1}$-measurable, and $\theta_i$ is the function of $(X_1,X_2,Y_2,\cdots,X_{i-1},Y_{i-1})$ for $i\geq1,$ that is,   there exist measurable functions $\{\pi_i\}_{i\geq2}$ such that
	\begin{align}
	\theta_2=\pi_2(X_1),\ \ \theta_i=\pi_i(X_1,X_2,Y_2,\cdots,X_{i-1},Y_{i-1}),~i\geq3.
	\end{align}
	
	Note that $\{X_1,Y_1,\cdots,X_{n},Y_{n}\}$ are independent under  hypothesis $H_1$. Applying \eqref{eq-3}, we have
	\begin{equation}\label{eq:Rx_1}
		\begin{split}
			&\xi_0E_1\biggl[\varphi\biggl(x+\sum_{i=1}^{n} Z_i^\theta\biggr)\biggr]\\
			=&\xi_0E_1\biggl[\varphi\biggl(x+X_1+\sum_{i=2}^{n}\big[\theta_iX_i+(1-\theta_i)Y_i\big] \biggr)\biggr]\\
			=&\int_{\mathbb R}\cdots\int_{\mathbb R} \xi_0 \varphi\biggl(x+u_1+
			\big[\pi_2(u_1)u_2+(1-\pi_2(u_1))s_2\big]\\
			+&\sum_{i=3}^{n}\big[\pi_i(u_1,\cdots,s_{i-1})u_i+(1-\pi_i(u_1,\cdots,s_{i-1}))s_i\big]\biggr)
			f_1(u_1)\ud u_1\prod^{n}_{i=2}f_{1}(u_i)f_2(s_i) \ud u_i \ud s_i\\	
			=&\int_{\mathbb R} E_1\biggl[\varphi\biggl(x+u_1+\sum_{i=2}^{n} Z_i^{\theta[u_1]}\biggr)\biggr]\xi_0 f_{1}(u_1)\ud u_1\\	
			=&\int_{\mathbb R} \xi_1^\theta(u_1) E_1\biggl[\varphi\biggl(x+u_1+\sum_{i=2}^{n} Z_i^{\theta[u_1]}\biggr)\biggr]\left(\xi_0f_{1}(u_1)+ (1-\xi_0)f_{2}(u_1)\right)\ud u_1.
		\end{split}
	\end{equation}
	Similar arguments show that
	\begin{equation}\label{eq:Rx_2}
		\begin{split}
			&(1-\xi_0)E_2\biggl[\varphi\biggl(x+\sum_{i=1}^{n} Z_i^\theta\biggr)\biggr]\\
			=&\int_{\mathbb R} \left(1-\xi_1^\theta(u_1)\right) E_2\biggl[\varphi\biggl(x+u_1+\sum_{i=2}^{n} Z_i^{\theta[u_1]}\biggr)\biggr]\left(\xi_0f_{1}(u_1)+(1-\xi_0)f_{2}(u_1)\right)\ud u_1.
		\end{split}
	\end{equation}
	Combining  \eqref{E_xi}, \eqref{eq:Rx_1} and \eqref{eq:Rx_2}, we find that
	\begin{equation}
		\begin{aligned}
		&E_{\xi_0}\biggl[\varphi\biggl(x+\sum_{i=1}^{n} Z_i^\theta\biggr)\biggr]\\
			=&\int_{\mathbb R} \xi_1^\theta(u_1)E_1\biggl[\varphi\biggl(x+u_1+\sum_{i=2}^{n} Z_i^{\theta[u_1]}\biggr)\biggr] \left(\xi_0f_{1}(u_1)+ (1-\xi_0)f_{2}(u_1)\right)\ud u_1\\
			 &+\int_{\mathbb R} (1-\xi_1^\theta(u_1))E_2\biggl[\varphi\biggl(x+u_1+\sum_{i=2}^{n} Z_i^{\theta[u_1]}\biggr)\biggr] \left(\xi_0f_{1}(u_1)+(1-\xi_0)f_{2}(u_1)\right)\ud u_1\\
			=&\int_{\mathbb R} E_{\xi_1^\theta(u_1)}\biggl[\varphi\biggl(x+u_1+\sum_{i=2}^{n} Z_i^{\theta[u_1]}\biggr)\biggr]\left(\xi_0f_{1}(u_1)+(1-\xi_0)f_{2}(u_1)\right)\ud u_1\\
		=&\xi_0\int_{\mathbb R} h(x,u_1)f_{1}(u_1)\ud u_1+ (1-\xi_0)\int_{\mathbb R} h(x,u_1)f_{2}(u_1)\ud u_1\\
		=& \xi_0E_1[h(x,X_1)]
		+(1-\xi_0)E_2[h(x,X_1)]\\
		=&E_{\xi_0}[h(x,X_1)].
	\end{aligned}
	\end{equation}
	This completes the proof.
\end{proof}

We can see that the form of $h(x,u)$ is very similar to the expected utility of the $(\xi_1^\theta(u),n-1,x+u)$-bandit with some strategy. Actually, there is indeed such a strategy to make the value of $h(x,u)$ equal to the expected utility of the $(\xi_1^\theta(u),n-1,x+u)$-bandit.
\begin{lem}
	\label{lem:NtoN_1}
	For any strategy $\theta\in \Theta_n$, let $h(x,u)=E_{\xi_1^\theta(u)}\left[\varphi\Bigl(x+u+\sum_{i=2}^{n} Z_i^{\theta[u]}\Bigr)\right]$, then for any $u$, there exists a strategy $\rho\in \Theta_{n-1}$, such that the value of $h(x,u)$ equals to the expected utility of the $(\xi_1^\theta(u),n-1,x+u)$-bandit with strategy $\rho$, i.e.
	\begin{align}
		h(x,u)=E_{\xi_1^\theta(u)}\left[\varphi\Bigl(x+u+\sum_{i=2}^{n} Z_i^{\theta[u]}\Bigr)\right]=E_{\xi_1^\theta(u)}\left[\varphi\Bigl(x+u+\sum_{i=1}^{n-1} Z_i^{\rho}\Bigr)\right].
	\end{align}
\end{lem}
\begin{proof}
	We know that $\theta[u]$ is obtained by fixing the payoff $u$ of the first trial, hence for any $\theta'_i\in \theta[u]$, it is $\sigma(X_2,Y_2,\dots,X_{i-1},Y_{i-1})$-measurable. Then there are measurable functions $\pi_i$, $i\ge 2$, such that
	\begin{align}
		\theta'_i=\pi_i(X_2,Y_2,\dots,X_{i-1},Y_{i-1}), i\ge 2.
	\end{align}
	Define a new strategy $\rho\in\Theta_{n-1}$ and let
	\begin{align}
		\rho_i=\pi_{i+1}(X_1,Y_1,\dots,X_{i-1},Y_{i-1}), 1\le i\le n-1.
	\end{align}
	By the definition of $\rho$, we know that $\rho_i$ has the same distribution with $\theta'_{i+1}$ in both hypotheses, so $Z_i^{\theta[u]}$ and $Z_i^{\rho}$ have the same distribution. Using this fact, we can easily varify that
	\begin{align}
		E_{\xi_1^\theta(u)}\left[\varphi\Bigl(x+u+\sum_{i=2}^{n} Z_i^{\theta[u]}\Bigr)\right]=E_{\xi_1^\theta(u)}\left[\varphi\Bigl(x+u+\sum_{i=1}^{n-1} Z_i^{\rho}\Bigr)\right].
	\end{align}
\end{proof}

% If you write a theorem, lemma, proposition etc please use the
% appropriate environments. For instance:

%\section{Dynamic Programming}%
%\label{sec:dynamic_programming}
\section{Main results}%
\label{sec:main_results}

In this section, we first investigate the dynamic programming property of the expected utility, which plays an important role in the subsequent arguments.
Similar results are found in many literatures on bandit problems, but only for the case of $\varphi(x)=x$ (e.g. \cite{Berry1985, Feldman1962}). Our result extends the classic ones.

\begin{thm}
	\label{thm:R*}
	For each $\xi_0\in [0,1]$, $n\geq1$ and $x\in\mathbb{R}$, consider the $(\xi_0,n,x)$-bandit. The optimal strategy $\theta^{[n]}\in \Theta_n$ exists. And there are measurable functions 
	\begin{align*}\pi^{[n]}_i:[0,1]\times\mathbb{R}\times\mathbb{R}^{2i-2}\mapsto \{0,1\},\quad 1\le i\le n, \end{align*}
	such that for any $\xi_0\in [0,1]$ and $x\in\mathbb{R}$, the optimal strategy $\theta^{[n]}$ for the $(\xi_0,n,x)$-bandit satisfies
	\begin{equation}\label{eq:OptSt}
		\begin{aligned}
			\theta^{[n]}_1&=\pi^{[n]}_1(\xi_0,x),\\
			\theta^{[n]}_i&=\pi^{[n]}_i(\xi_0,x,X_1,Y_1,\dots,X_{i-1},Y_{i-1}),~i\geq2.
		\end{aligned}
	\end{equation}
	And the optimal expected utility satisfies the following dynamic programming property
	\begin{equation}\label{eq:DyProg}
		\begin{aligned}
			V(\xi_0,n,x)&= \sup_{\theta\in \Theta_n} E_{\xi_0}\left[V\left(\xi_1^\theta,n-1,x+Z^\theta_1\right)\right]\\
			&= \max \left\{ E_{\xi_0}\left[V\biggl(\frac{\xi_0f_1(X_1) }{\xi_0 f_1(X_1)+(1-\xi_0)f_2(X_1)},n-1,x+X_1\biggr)\right],\right.\\
			&~~~~~~~~~~~\left.E_{\xi_0}\left[V\biggl(\frac{\xi_0f_2(Y_1) }{\xi_0 f_2(Y_1)+(1-\xi_0)f_1(Y_1)},n-1,x+Y_1\biggr)\right] \right\}.
		\end{aligned}
	\end{equation}
\end{thm}
\begin{proof}
	Firstly, let's prove that the left side in \eqref{eq:DyProg} is less or equal to the right side.
	By equation \eqref{eq:V_defn},
	\begin{align*}
		&V(\xi_0,n,x)=\sup_{\theta\in \Theta_n} E_{\xi_0}\biggl[\varphi\biggl(x+\sum_{i=1}^n Z^\theta_i\biggr)\biggr].%{\\
		%=& \max\left\{ \sup_{S\in \mathcal{D}_n,s_1=1} E_{\xi_0}\left[\varphi\left(x+X_1+\sum_{i=2}^n Z^S_i\right)\right], \sup_{S\in \mathcal{D}_n,s_1=0} E_{\xi_0}\left[\varphi\left(x+Y_1+\sum_{i=2}^n Z^S_i\right)\right]\right\}.
	\end{align*}
	Applying Lemma \ref{lem:Exp}, there is
	\begin{align}
		V(\xi_0,n,x)
		= \sup_{\theta\in \Theta_n} E_{\xi_0}\left[h^\theta(x,Z_1^\theta)\right],
	\end{align}
	where $h^\theta(x,u)=E_{\xi_1^\theta(u)}\left[\varphi\Bigl(x+u+\sum_{i=2}^{n} Z_i^{\theta[u]}\Bigr)\right].$
	Note that here the conditional strategy $\theta[u]$ is obtained by fixing the first payoff to $u$ for strategy $\theta$. For any fixed $\theta$ and $u$, by Lemma \ref{lem:NtoN_1}, there exists a strategy $\rho\in \Theta_{n-1}$ that has the same distribution with $\theta[u]$ and statisfies
	\begin{align}
		E_{\xi_1^\theta(u)}\biggl[\varphi\biggl(x+u+\sum_{i=2}^{n} Z_i^{\theta[u]}\biggr)\biggr]=E_{\xi_1^\theta(u)}\biggl[\varphi\biggl(x+u+\sum_{i=1}^{n-1} Z_i^{\rho}\biggr)\biggr].
\end{align}
	Hence we get
	\begin{align}
		h^\theta(x,u)\le \sup_{\rho\in \Theta_{n-1}} E_{\xi_1^\theta(u)}\biggl[\varphi\biggl(x+u+\sum_{i=1}^{n-1} Z_i^{\rho}\biggr)\biggr]=V\left(\xi_1^\theta(u),n-1,x+u\right).
	\end{align}
	Then we obtain the following inequality,
	\begin{equation}
		\begin{aligned}
		V(\xi_0,n,x)&\le \sup_{\theta\in \Theta_n} E_{\xi_0}\left[V\left(\xi_1^\theta,n-1,x+Z^\theta_1\right)\right]\\
			    &= \max \left\{ E_{\xi_0}\left[V\biggl(\frac{\xi_0f_1(X_1) }{\xi_0 f_1(X_1)+(1-\xi_0)f_2(X_1)},n-1,x+X_1\biggr)\right],\right.\\
		&~~~~~~~~~~~\left.E_{\xi_0}\left[V\biggl(\frac{\xi_0f_2(Y_1) }{\xi_0 f_2(Y_1)+(1-\xi_0)f_1(Y_1)},n-1,x+Y_1\biggr)\right] \right\}.
	\end{aligned}
	\end{equation}
	The equality here is due to the fact that there are only two possible forms of the posterior probability $\xi_1^\theta$, see equation \eqref{eq-3}.
	
	Now we prove the reverse inequality and the existence of the optimal strategy by mathematical induction.
	
	When $n=1$, define strategy $\theta^{[1]}$ by
	\begin{equation}
	\begin{aligned}
		\theta^{[1]}_1=&\pi^{[1]}_1(\xi_0,x)\\:=& \begin{cases}
			1,\quad \text{if $(2\xi_0-1)\Bigl[ \int\varphi(x+u)f_1(u)\ud u-\int\varphi(x+u)f_2(u)\ud u\Bigr]\ge 0,$}\\
			0,\quad \text{if $(2\xi_0-1)\Bigl[ \int\varphi(x+u)f_1(u)\ud u-\int\varphi(x+u)f_2(u)\ud u\Bigr]< 0.$}
		\end{cases}
	\end{aligned}
\end{equation}
	By the measurability of $\varphi$, it can be easily varified that $\pi^{[1]}_1$ is measurable and $\theta^{[1]}$ is the optimal strategy of $(\xi_0,1,x)$-bandit for any $\xi_0\in[0,1]$ and $x\in \mathbb{R}$.
	
	Now we assume that, for any $\xi_0\in[0,1]$ and $x\in \mathbb{R}$, there exists optimal strategy $\theta^{[n-1]}\in \Theta_{n-1}$, and measurable functions $\bigl\{\pi^{[n-1]}_i\bigr\}_{1\le i\le n-1}$ for $(\xi_0,n-1,x)$-bandit  satisfy the condition \eqref{eq:OptSt}. We will find the optimal strategy for the $(\xi_0,n,x)$-bandit.
	
	Define strategy $\bar\theta=\{\bar \theta_1,\dots,\bar \theta_n\} \in \Theta_n$ by
	\begin{align*}
		\bar \theta_1&=1,\\
		\bar \theta_i&=\pi^{[n-1]}_{i-1}(\xi^{\bar \theta}_1(X_1),x+X_1,X_2,Y_2,\dots,X_{i-1},Y_{i-1}), \text{ for $2\le i\le n$,}
\end{align*}
	where $\xi^{\bar \theta}_1(X_1)=\frac{\xi_0f_1(X_1) }{\xi_0 f_1(X_1)+(1-\xi_0)f_2(X_1)}$.

	Then by Lemma \ref{lem:NtoN_1}, for any $u\in\mathbb{R}$, there is
	\begin{align*}
		E_{\xi^{\bar{\theta}}_1(u)}\biggl[\varphi\biggl(x+u+\sum_{i=2}^{n}Z^{{\bar{\theta}}[u]}_i\biggr)\biggr]
		=&E_{\xi^{\bar{\theta}}_1(u)}\biggl[\varphi\biggl(x+u+\sum_{i=1}^{n-1}Z^{\theta^{[n-1]}}_i\biggr)\biggr]\\
		=&V\bigg(\frac{\xi_0f_1(u) }{\xi_0 f_1(u)+(1-\xi_0)f_2(u)}, n-1, x+u\bigg).
	\end{align*}
	Then by Lemma \ref{lem:Exp}, we obtain
	\begin{align}
		W(\xi_0,n,x,\bar{\theta})
		=&E_{\xi_0}\biggl[\varphi\biggl(x+\sum_{i=1}^n Z^{\bar{\theta}}\biggr)\biggr]\notag\\
		=&E_{\xi_0}\biggl[V\biggl(\frac{\xi_0f_1(X_1) }{\xi_0 f_1(X_1)+(1-\xi_0)f_2(X_1)}, n-1, x+X_1\biggr)\biggr].
	\end{align}
	
	Similarly, we can define $\widetilde \theta=\{\widetilde \theta_1,\dots,\widetilde \theta_n\} \in \Theta_n$ by
	\begin{align*}
		\widetilde \theta_1&=0,\\
		\widetilde \theta_i&=\pi^{[n-1]}_{i-1}(\xi^{\widetilde \theta}_1(Y_1),x+Y_1,X_2,Y_2,\dots,X_{i-1},Y_{i-1}), \text{ for $2\le i\le n$,}
	\end{align*}
	where $\xi^{\widetilde \theta}_1(Y_1)=\frac{\xi_0f_2(Y_1) }{\xi_0 f_2(Y_1)+(1-\xi_0)f_1(Y_1)}$.

	We can also get
	\begin{align}
		W(\xi_0,n,x,\widetilde \theta)
		=&E_{\xi_0}\biggl[\varphi\biggl(x+\sum_{i=1}^n Z^{\widetilde \theta}\biggr)\biggr]\notag\\
		=&E_{\xi_0}\biggl[V\biggl(\frac{\xi_0f_2(Y_1) }{\xi_0 f_2(Y_1)+(1-\xi_0)f_1(Y_1)}, n-1, x+Y_1\biggr)\biggr].
	\end{align}
	
	By the definition of $V(\xi_0,n,x)$, we have
	\begin{equation}
	\begin{aligned}
		V(\xi_0,n,x)&\ge
		\max\left\{W(\xi_0,n,x,\bar{\theta}), W(\xi_0,n,x,\widetilde \theta)\right\}\\
		&= \max \left\{ E_{\xi_0}\biggl[V\bigg(\frac{\xi_0f_1(X_1) }{\xi_0 f_1(X_1)+(1-\xi_0)f_2(X_1)},n-1,x+X_1\bigg)\biggr],\right.\\
		&~~~~~~~~~~~\left.E_{\xi_0}\biggl[V\bigg(\frac{\xi_0f_2(Y_1) }{\xi_0 f_2(Y_1)+(1-\xi_0)f_1(Y_1)},n-1,x+Y_1\bigg)\biggr] \right\}.
	\end{aligned}
	\end{equation}
	Now the reverse inequality holds. If $\bar \theta$ achieves the optimal expected utility, then let $\theta^{[n]}=\bar \theta$; and if $\widetilde \theta$ achieves the optimal expected utility, let $\theta^{[n]}=\widetilde \theta$. Hence $\theta^{[n]}$ is clearly the optimal strategy for the $(\xi,n,x)$-bandit.

	Using the definition, we can easily varify that both $W(\xi_0,n,x,\bar \theta)$ and $W(\xi_0,n,x,\widetilde \theta)$ are measurable functions of $(\xi_0,x)$. This fact, along with the measureability of $\pi^{[n-1]}$, $\xi_1^{\bar \theta}(u)$ and $\xi_1^{\widetilde \theta}(u)$, guarantees the existence and measureability of$\{\pi^{[n]}_i\}_{1\le i\le n}$.
\end{proof}

Now we are going to study the specific form of the optimal strategy, for a finite two-armed bandit with utility function $\varphi$. It is obvious that different utility functions may lead to different optimal strategies, but we will show that when a reasonable condition is satisfied, the optimal strategy is independent of the specific form of $\varphi$.

Recall that the myopic strategy for $(\xi_0,n,x)$-bandits is $\mathsf{M}^n=\{\mathsf{m}^n_1,\cdots,\mathsf{m}^n_n\}$: in trial $i$, play $X$-arm if the posterior probability $\xi_{i-1}^{\mathsf{M}^n}\ge \frac{1}{2}$, or $Y$-arm if $\xi_{i-1}^{\mathsf{M}^n}<\frac{1}{2}$. Note that $\mathsf{M}^n\in\Theta_n$. In fact, $\mathsf{M}^n$ can be denoted by
\begin{equation}\begin{split}\label{Feldman Stra}
		\mathsf{M}^n=&\{\mathsf{m}^n_1,\cdots,\mathsf{m}^n_n\}\\
		=&\left\{g\bigl(\xi_0\bigr),g\bigl(\xi_1^{\mathsf{M}^n}\bigr),\cdots,g\bigl(\xi_{n-1}^{\mathsf{M}^n}\bigr)\right\}\in\Theta_n,
\end{split}\end{equation}
where $g(x)=1,\text{ if $x\geq\frac{1}{2}$,}$ or $g(x)=0,\text{if $x<\frac{1}{2}$}$. From the definition of $\xi_{i-1}^{\mathsf{M}^n}$ and $\mathsf{m}^n_i$, we know that they are both independent of $n$ and $x$. Hence we can write $\xi_{i-1}^{\mathsf{M}^n}$ shortly as $\xi_{i-1}^{\mathsf{M}}$, and write $\mathsf{m}^n_i$ shortly as $\mathsf{m}_i$. Then the myopic strategy $\mathsf{M}^n$ is now denoted by
\begin{equation}
	\begin{aligned}
	\mathsf{M}^n=&\{\mathsf{m}_1,\cdots,\mathsf{m}_n\}\\
	=&\left\{g\bigl(\xi_0\bigr),g\bigl(\xi_1^{\mathsf{M}}\bigr),\cdots,g\bigl(\xi_{n-1}^{\mathsf{M}}\bigr)\right\}.
	\end{aligned}
\end{equation}
%\begin{equation}\label{eq:g}
%	g(x)=\left\{\begin{aligned}
%		&1,&&\text{if $x\geq\frac{1}{2}$,}\\
%		&0,&&\text{if $x<\frac{1}{2}$,}
%	\end{aligned}\right.
%\end{equation}
%then
Next, we will give a condition on $\varphi$ which is necessary and sufficient for $\mathsf{M}^n$ being the optimal strategy.
\begin{thm}
	\label{thm:main}
	For any integer $n\ge 1$, the myopic strategy $\mathsf{M}^n$ is the optimal strategy of the $(\xi_0,n,x)$-bandit for $\forall \xi_0\in[0,1]$, $\forall x\in \mathbb{R}$,  if and only if
	\begin{equation}
		\label{conditionI}
		E_1[\varphi(u+X_1)]\ge E_1[\varphi(u+Y_1)], \quad \text{for }\,\forall u\in \mathbb{R}.\tag{I}
	\end{equation}
\end{thm}
\begin{rem}
	When $\varphi(x)=x$, the condition (I) is actually
	\begin{align}E_1[X_1]\ge E_1[Y_1],\end{align}
	and Theorem \ref{thm:main} in this case is exactly Theorem 2.1 in Feldman\cite{Feldman1962}.
\end{rem}
\begin{rem}
	When the two distributions $F_1$, $F_2$ are Bernoulli distributions, say, Bernoulli($\alpha$) and Bernoulli($\beta$), $\alpha,\beta\in(0,1)$, then the condition \eqref{conditionI} is written as
	\begin{equation}
		(\varphi(x+1)-\varphi(x))(\alpha-\beta)\ge 0, \text{ for any $x\in \mathbb{R}$}.
	\end{equation}
	If $\varphi(x)$ is an increasing function of $x$, and $\alpha\ge \beta$, then the condition \eqref{conditionI} holds.
\end{rem}
\begin{rem}
	Note that the condition \eqref{conditionI} here is a necessary and sufficient condition to make $\mathsf{M}^n$ the optimal strategy of the $(\xi_0,n,x)$-bandit model for any $x\in \mathbb{R}$ and any $\xi_0\in[0,1]$. However, when the condition \eqref{conditionI} is not satisfied, it is still possible that there is a specific triple $(\overline{\xi_0},\overline{n},\overline{x})$ that makes $\mathsf{M}^{\overline{n}}$ the optimal strategy of the $(\overline{\xi_0},\overline{n},\overline{x})$-bandit, but the optimality of $\mathsf{M}^n$ does not hold for general $(\xi_0,n,x)$ triples. 
\end{rem}

To achieve our goal, we need to formulate some properties of the expected utility of $\mathsf{M}^n$, which can be considered as extensions of \textit{Properties of $R_N$}, and Lemma 2.1 in Feldman's paper\cite{Feldman1962}.
\begin{lem}
	\label{lem:R^*N}
	For each $n\in\mathbb{Z}^+$, we have
	
	1) The expected utility with strategy $\mathsf{M}^n$ is symmetric about $\xi_0=\frac{1}{2}$, i.e.
%	\begin{align}
%		E_{\xi_0}\biggl[\varphi\biggl(x+\sum_{i=1}^{n} Z_i^{\mathsf{M}^n}\biggr)\biggr]=E_{1-\xi_0}\biggl[\varphi\biggl(x+\sum_{i=1}^{n} Z_i^{\mathsf{M}^n}\biggr)\biggr], ~\forall \xi_0\in [0,1].
%	\end{align}
	\begin{align}
		W(\xi_0,n,x,\mathsf{M}^n)=W(1-\xi_0,n,x,\mathsf{M}^n), ~\forall \xi_0\in [0,1], \forall x\in\mathbb{R}.
	\end{align}
	
	2) Let strategies $\mathsf{L}^n=\{1,g(\xi_1^{\mathsf{L}^n}),\cdots,g(\xi_{n-1}^{\mathsf{L}^n})\}$ and $\mathsf{R}^n=\{0,g(\xi_1^{\mathsf{R}^n}),\cdots,g(\xi_{n-1}^{\mathsf{R}^n})\}$, where $g(\cdot)$ is the function defined in \eqref{Feldman Stra}, then
%	\begin{align}\label{eq:S^1=S^2}
%		E_{\xi_0}\biggl[\varphi\biggl(x+\sum_{i=1}^{n} Z_i^{r_1}\biggr)\biggr]=E_{1-\xi_0}\biggl[\varphi\biggl(x+\sum_{i=1}^{n} Z_i^{r_2}\biggr)\biggr], ~\forall \xi_0\in [0,1].
%	\end{align}
	\begin{align}\label{eq:S^1=S^2}
		W(\xi_0,n,x,\mathsf{L}^n)=W(1-\xi_0,n,x,\mathsf{R}^n), ~\forall \xi_0\in [0,1], \forall x\in\mathbb{R}.
	\end{align}
	%	
	%	3) If Assumption (I) holds, we have
	%			\begin{equation*}
	%				V(0,N,x)=\idotsint \varphi(x+t_1+\dots+t_N)f_1(t_1)\dots f_1(t_N)\ud t_1\dots\ud t_N.
	%			\end{equation*}
\end{lem}
\begin{proof}
	By the definitions of $\mathsf{M}^n$, $\mathsf{L}^n$, and $\mathsf{R}^n$, we can easily get that $\mathsf{M}^n=\mathsf{L}^n$ when $\xi_0\ge 1/2$, and $\mathsf{M}^n=\mathsf{R}^n$ when $\xi_0< 1/2$. Hence we have the following relations for their expected utilities.

	\begin{align}
		W(\xi_0,n,x,\mathsf{M}^n)&=\begin{cases}
			W(\xi_0,n,x,\mathsf{L}^n),\quad\text{if $\xi_{0}\geq\frac{1}{2}$,}\\
			W(\xi_0,n,x,\mathsf{R}^n),\quad\text{otherwise,}
		\end{cases}\label{eq:Wcase}\\
		W(1-\xi_0,n,x,\mathsf{M}^n)&=\begin{cases}
			W(1-\xi_0,n,x,\mathsf{R}^n),\quad\text{if $\xi_{0}\geq\frac{1}{2}$,}\\
			W(1-\xi_0,n,x,\mathsf{L}^n),\quad\text{otherwise.}
		\end{cases}
	\end{align}
	So we only need to prove \eqref{eq:S^1=S^2}.
	
	Let us use mathematical induction. Starting with $n=1$, we have
	\begin{align*}
		E_{\xi_0}\bigl[\varphi\bigl(x+ Z_1^{r_1}\bigr)\bigr]
		&=E_{\xi_0}[\varphi(x+X_1)]\\
		&=\xi_{0}\int_{\mathbb R}\varphi(x+u)f_1(u)\ud u+(1-\xi_{0})\int_{\mathbb R} \varphi(x+u)f_2(u)\ud u\\
		&=E_{1-\xi_0}[\varphi(x+Y_1)]\\
		&=E_{1-\xi_0}\bigl[\varphi\bigl(x+ Z_1^{r_2}\bigr)\bigr].
	\end{align*}
	Then the desired equation $W(\xi_0,1,x,\mathsf{L}^1)=W(1-\xi_0,1,x,\mathsf{R}^1)$ holds.

	%	Then recall that $V(\xi,1,x)=\max \{W(\xi,1,x,\pi^X_1),W(\xi,1,x,\pi^Y_1)\}$, so properties (1) and (2) are clearly satisfied. By assumption (I) and the optimality, we have
	%	\begin{equation*}
	%		V(0,1,x)=\max \left\{\int \varphi(x+t)f_2(t)\ud t,\int\varphi(x+t)f_1(t)\ud t\right\},
	%	\end{equation*}
	%	hence $V(0,1,x)=\int\varphi(x+t)f_1(t)\ud t$ holds.
	Assume that \eqref{eq:S^1=S^2} holds for $n=k$, that is for each $\xi_0\in [0,1]$,and $ x\in\mathbb{R}$,
%	\begin{align}
%		E_{\xi_0}\biggl[\varphi\biggl(x+\sum_{i=1}^{k} Z_i^{r_1}\biggr)\biggr]=E_{1-\xi_0}\biggl[\varphi\biggl(x+\sum_{i=1}^{k} Z_i^{r_2}\biggr)\biggr],
%	\end{align}
	\begin{align}
		W(\xi_0,k,x,\mathsf{L}^k)=W(1-\xi_0,k,x,\mathsf{R}^k), ~\forall \xi_0\in [0,1], \forall x\in\mathbb{R},
	\end{align}
	and
	\begin{align}
%E_{\xi_0}\biggl[\varphi\biggl(x+\sum_{i=1}^{k} Z_i^{\mathsf{M}}\biggr)\biggr]=E_{1-\xi_0}\biggl[\varphi\biggl(x+\sum_{i=1}^{k} Z_i^{\mathsf{M}}\biggr)\biggr].
		W(\xi_0,k,x,\mathsf{M}^k)=W(1-\xi_0,k,x,\mathsf{M}^k), ~\forall \xi_0\in [0,1], \forall x\in\mathbb{R}.
	\end{align}
	For $n=k+1$, we can obtain the following relation by using Lemma \ref{lem:Exp}.
	\begin{equation}\label{eq:LeqR_1}
	\begin{aligned}
		&E_{\xi_0}\biggl[\varphi\biggl(x+\sum_{i=1}^{k+1} Z_i^{\mathsf{L}^{k+1}}\biggr)\biggr]\\
		=&\int_{\mathbb R} E_{\xi_1^{\mathsf{L}^{k+1}}(u)}\biggl[\varphi\biggl(x+u+\sum_{i=2}^{k+1} Z_i^{\mathsf{L}^{k+1}[u]}\biggr)\biggr]\left(\xi_0f_{1}(u)+(1-\xi_0)f_{2}(u)\right)\ud u.
	\end{aligned}
	\end{equation}
	Recall that $\mathsf{L}^{k+1}[u]$ is obtained from $\mathsf{L}^{k+1}$ by fixing the payoff of the first trial to be $u$. By Lemma \ref{lem:NtoN_1} and the definitions of $\mathsf{L}^{k+1}$ and $\mathsf{M}^k$, it can be easily varified that
	\begin{equation}\label{eq:LeqR_2}
		\begin{aligned}	
		E_{\xi_1^{\mathsf{L}^{k+1}}(u)}\biggl[\varphi\biggl(x+u+\sum_{i=2}^{k+1} Z_i^{\mathsf{L}^{k+1}[u]}\biggr)\biggr]&=E_{\xi_1^{\mathsf{L}^{k+1}}(u)}\biggl[\varphi\biggl(x+u+\sum_{i=1}^{k} Z_i^{\mathsf{M}^{k}}\biggr)\biggr]\\
																&=W(\xi_1^{\mathsf{L}^{k+1}}(u),k,x+u,\mathsf{M}^{k}).
	\end{aligned}
	\end{equation}
	And by the properties of posterior probability, we have
	\begin{equation}\label{eq:LeqR_3}
	\begin{aligned}
\xi_1^{\mathsf{L}^{k}}(u)=
		\xi_1^{\mathsf{L}^{k+1}}(u)=&\frac{\xi_0f_1(u)}{\xi_0f_1(u)+(1-\xi_0)f_2(u)}.
	\end{aligned}
	\end{equation}
	Combining the above equations \eqref{eq:LeqR_1}, \eqref{eq:LeqR_2} and \eqref{eq:LeqR_3}, we can get
	\begin{equation}\label{eq:LeqR_4}
	\begin{aligned}
		&E_{\xi_0}\biggl[\varphi\biggl(x+\sum_{i=1}^{k+1} Z_i^{\mathsf{L}^{k+1}}\biggr)\biggr]\\
		=&\int_{\mathbb R} W\Bigl(\frac{\xi_0f_1(u)}{\xi_0f_1(u)+(1-\xi_0)f_2(u)},k,x+u,\mathsf{M}^{k}\Bigr)\left(\xi_0f_{1}(u)+(1-\xi_0)f_{2}(u)\right)\ud u.
	\end{aligned}
	\end{equation}
	A similar discussion gives
	\begin{equation}\label{eq:LeqR_5}
	\begin{aligned}
		&E_{1-\xi_0}\biggl[\varphi\biggl(x+\sum_{i=1}^{k+1} Z_i^{\mathsf{R}^{k+1}}\biggr)\biggr]\\
		=&\int_{\mathbb R} W\Bigl(\frac{(1-\xi_0)f_2(u)}{\xi_0f_1(u)+(1-\xi_0)f_2(u)},k,x+u,\mathsf{M}^{k}\Bigr)\left(\xi_0f_{1}(u)+(1-\xi_0)f_{2}(u)\right)\ud u.
	\end{aligned}
	\end{equation}
	Applying the induction hypothesis, we get
	\begin{equation}
		\begin{aligned}
		%E_{\xi_1^{r_1}(u)}\biggl[\varphi\biggl(x+u+\sum_{i=1}^{k} Z_i^{\mathsf{M}}\biggr)\biggr]=&E_{1-\xi_1^{r_1}(u)}\biggl[\varphi\biggl(x+u+\sum_{i=1}^{k} Z_i^{\mathsf{M}}\biggr)\biggr]\\
		%=&E_{\xi_1^{r_2}(u)}\biggl[\varphi\biggl(x+u+\sum_{i=1}^{k} Z_i^{\mathsf{M}}\biggr)\biggr].
		  &W\Bigl(\frac{\xi_0f_1(u)}{\xi_0f_1(u)+(1-\xi_0)f_2(u)},k,x+u,\mathsf{M}^{k}\Bigr)\\
		=&W\Bigl(\frac{(1-\xi_0)f_2(u)}{\xi_0f_1(u)+(1-\xi_0)f_2(u)},k,x+u,\mathsf{M}^{k}\Bigr).
	\end{aligned}
	\end{equation}
	Now we obtain that
	\begin{equation}
	\begin{aligned}
		W(\xi_0,k+1,x,\mathsf{L}^{k+1})
		=&E_{\xi_0}\biggl[\varphi\biggl(x+\sum_{i=1}^{k+1} Z_i^{\mathsf{L}^{k+1}}\biggr)\biggr]\\
		=&E_{1-\xi_0}\biggl[\varphi\biggl(x+\sum_{i=1}^{k+1} Z_i^{\mathsf{R}^{k+1}}\biggr)\biggr]\\
		=&W(1-\xi_0,k+1,x,\mathsf{R}^{k+1}).
	\end{aligned}
	\end{equation}
	This completes the proof.

\end{proof}
\begin{lem}
	\label{lem:XY=YX}
	Consider the following strategies:

	For $n=2$, let $\mathsf{U}^2=\{1,0\}$ and $\mathsf{V}^2=\{0,1\}$.

	For each $n\geq3$, let
\begin{align*}
	\mathsf{U}^n&=\{1,0,g(\xi_2^{\mathsf{U}^n}),\cdots,g(\xi_{n-1}^{\mathsf{U}^n})\},\\
	\mathsf{V}^n&=\{0,1,g(\xi_2^{\mathsf{V}^n}),\cdots,g(\xi_{n-1}^{\mathsf{V}^n})\},
\end{align*}
where $g(\cdot)$ is the function defined in \eqref{Feldman Stra}. Then for each $n\ge 2$, $\xi_0\in[0,1]$ and $x\in\mathbb{R}$, the expected utilities obtained by using $\mathsf{U}^n$ and $\mathsf{V}^n$ satisfy the relation
	\begin{align}\label{eq:T^12=T^21}
		W(\xi_0,n,x,\mathsf{U}^n)=W(\xi_0,n,x,\mathsf{V}^n).
		%E_{\xi_0}\biggl[\varphi\biggl(x+\sum_{i=1}^{n} Z_i^{\theta_1}\biggr)\biggr]=E_{\xi_0}\biggl[\varphi\biggl(x+\sum_{i=1}^{n} Z_i^{\theta_2}\biggr)\biggr].
	\end{align}
\end{lem}
\begin{proof}
	For $n=2$ case, the result is obvious. We only need to consider cases where $n\ge 3$.
	Let $\xi_2^{\mathsf{U}^n}(u,s)$ be the posterior probability given the first payoff $u$ and second payoff $s$, when using strategy $\mathsf{U}^n$; and $\xi_2^{\mathsf{V}^n}(s,u)$ be the posterior probability given the first payoff $s$ and second payoff $u$, when using strategy $\mathsf{V}^n$. According to \eqref{posterior}, there is
	\begin{align}\label{eq:XY-1}
		\xi_2^{\mathsf{U}^n}(u,s)=&\frac{\xi_0 f_1(u)f_2(s)}{\xi_0f_1(u)f_2(s)+(1-\xi_0)f_2(u)f_1(s)}
		=\xi_2^{\mathsf{V}^n}(s,u).
	\end{align}

	Let conditional strategy $\mathsf{U}^n[u,s]$ denotes the strategy $\mathsf{U}^n$ given the first payoff $u$ and second payoff $s$, and the conditional strategy $\mathsf{V}^n[s,u]$ denotes the strategy $\mathsf{V}^n$ given the first payoff $s$ and second payoff $u$. Then by \eqref{eq:XY-1}, and the definitions of $\mathsf{U}^n[u,s]$ and $\mathsf{V}^n[s,u]$, we can easily get that these two conditional strategies are the same for $i\ge 3$ trials.

	Using the same techniques in Lemma \ref{lem:Exp}, we obtain
	\begin{align}
		&W(\xi_0,n,x,\mathsf{U}^n)
		=E_{\xi_0}\biggl[\varphi\biggl(x+\sum_{i=1}^{n} Z_i^{\mathsf{U}^n}\biggr)\biggr]\notag\\
		=&\int_{\mathbb R}\int_{\mathbb R} E_{\xi_2^{\mathsf{U}^n}\!(u,s)\!}\biggl[\varphi\biggl(\!x+u+s+\sum_{j=3}^{n} Z_j^{\mathsf{U}^n[u,s]}\!\biggr)\!\biggr]
		\!\left(\xi_0f_{1}(u)f_{2}(s)+(1-\xi_0)f_{2}(u)f_{1}(s)\right)\ud s\ud u,\\
		\intertext{and}
		&W(\xi_0,n,x,\mathsf{V}^n)
		=E_{\xi_0}\biggl[\varphi\biggl(x+\sum_{i=1}^{n} Z_i^{\mathsf{V}^n}\biggr)\biggr]\notag\\
		=&\int_{\mathbb R}\int_{\mathbb R} E_{\xi_2^{\mathsf{V}^n}\!(s,u)\!}\biggl[\varphi\biggl(\!x+s+u+\sum_{j=3}^{n} Z_j^{\mathsf{V}^n[s,u]}\!\biggr)\!\biggr]
		\!\left(\xi_0f_{2}(s)f_{1}(u)+(1-\xi_0)f_{1}(s)f_{2}(u)\right)\ud u\ud s.
	\end{align}
	Then the desired result is obtained by using \eqref{eq:XY-1} and the fact that $\mathsf{U}^n[u,s]=\mathsf{V}^n[s,u]$ for $i\ge 3$ trials.
\end{proof}

%New proof
For each $n\in\mathbb{Z}^+$, let $\mathsf{L}^n$ and $\mathsf{R}^n$ be the strategies defined in Lemma \ref{lem:R^*N}. For $x\in\mathbb{R}$, $\xi_0\in[0,1]$, define the difference of expected utilities by
\begin{align*}
	\Delta_n(x,\xi_0):=&W(\xi_0,n,x,\mathsf{L}^n)-W(\xi_0,n,x,\mathsf{R}^n)\\
	=&E_{\xi_0}\biggl[\varphi\biggl(x+\sum_{i=1}^{n} Z_i^{\mathsf{L}^n}\biggr)\biggr] - E_{\xi_0}\biggl[\varphi\biggl(x+\sum_{i=1}^{n} Z_i^{\mathsf{R}^n}\biggr)\biggr].
\end{align*}
By Lemma \ref{lem:R^*N}, we can obtain that $\Delta_n(x,\xi_0)=-\Delta_n(x,1-\xi_0)$ and $\Delta_n(x,0.5)=0$. 

In the case that $\varphi(x)=x$, there is an important recurrence formula for $\Delta_n(x,\xi_0)$, see equation (12), (13) and (14) in \cite{Feldman1962} and equation (7.1.1) in \cite{Berry1985}. The following lemma shows that this recurrence formula still holds for a general utility function.
\begin{lem}
	\label{lem:DeltaNtoN_1}
	For $n\ge 2$, and any $x\in\mathbb{R}$, $\xi_0\in[0,1]$, there is
	\begin{equation}
	\begin{aligned}
		&\Delta_n(x,\xi_0)\\
		=&\int_{\mathbb R} I_{\bigl\{\xi_1^X(u)\ge 0.5\bigr\}}\Delta_{n-1}\bigl(x+u,\xi_1^X(u)\bigr)\bigl(\xi_0f_1(u)+(1-\xi_0)f_2(u)\bigr)\ud u,\\
		&+\int_{\mathbb R} I_{\bigl\{\xi_1^Y(u)<0.5\bigr\}}\Delta_{n-1}\bigl(x +u,\xi_1^Y(u)\bigr)\bigl(\xi_0f_2(u) +(1-\xi_0)f_1(u) \bigr)\ud u,
	\end{aligned}
	\end{equation}
	where $\xi_1^X(u)=\frac{\xi_0 f_1(u) }{\xi_0 f_1(u) +(1-\xi_0)f_2(u)}$ and $\xi_1^Y(u)=\frac{\xi_0 f_2(u) }{\xi_0 f_2(u) +(1-\xi_0)f_1(u)}$.
\end{lem}
\begin{proof}
	Consider strategies $\mathsf{U}^n,\mathsf{V}^n$ defined in Lemma \ref{lem:XY=YX}. Using Lemma \ref{lem:Exp} and Lemma \ref{lem:NtoN_1}, we can make the following two differences
	\begin{equation}
		\begin{aligned}
			&W(\xi_0,n,x,\mathsf{L}^n)-W(\xi_0,n,x,\mathsf{U}^n)\\
			=&E_{\xi_0}\biggl[\varphi\biggl(x+\sum_{i=1}^{n} Z_i^{\mathsf{L}^n}\biggr)\biggr] - E_{\xi_0}\biggl[\varphi\biggl(x+\sum_{i=1}^{n} Z_i^{\mathsf{U}^n}\biggr)\biggr]\\
			=&\int_{\mathbb R} I_{\bigl\{\xi_1^X(u)\ge 0.5\bigr\}}\Delta_{n-1}\bigl(x+u,\xi_1^X(u)\bigr)\bigl(\xi_0f_1(u)+(1-\xi_0)f_2(u)\bigr)\ud u,
		\end{aligned}
	\end{equation}
	\begin{equation}
		\begin{aligned}
			&W(\xi_0,n,x,\mathsf{V}^n)-W(\xi_0,n,x,\mathsf{R}^n)\\
			=&E_{\xi_0}\biggl[\varphi\biggl(x+\sum_{i=1}^{n} Z_i^{\mathsf{V}^n}\biggr)\biggr] - E_{\xi_0}\biggl[\varphi\biggl(x+\sum_{i=1}^{n} Z_i^{\mathsf{R}^n}\biggr)\biggr]\\
			=&\int_{\mathbb R} I_{\bigl\{\xi_1^Y(u)<0.5\bigr\}}\Delta_{n-1}\bigl(x +u,\xi_1^Y(u)\bigr)\bigl(\xi_0f_2(u) +(1-\xi_0)f_1(u) \bigr)\ud u.
		\end{aligned}
	\end{equation}
	The computational details are similar to the proofs of Lemma \ref{lem:R^*N} and Lemma \ref{lem:XY=YX}, so we omit them. 
	The desired formula is obtained by adding the above two equations and using Lemma \ref{lem:XY=YX}.
\end{proof}

The key to achieve our main theorem is to prove that for any fixed $x\in\mathbb{R}$ and $n\in\mathbb{Z}^+$, the above difference $\Delta_n(x,\xi_0)$ is an increasing function of $\xi_0$. To prove this assertion, we use a method similar to that used in Rodman \cite{Rodman1978}.

Define functions $D_n(x,t_X,t_Y), n=1,2,\dots$, for every tuple $(x,t_X,t_Y)$ of numbers, such that $x\in\mathbb{R}$ and $t_X,t_Y\ge 0$:
\begin{align}
	D_n(x,t_X,t_Y)= \begin{cases}
		(t_X+t_Y)\Delta_n(x,\frac{t_X}{t_X+t_Y})\quad &\text{if $t_X+t_Y>0$;}\\
		0&\text{otherwise.}
	\end{cases}
\end{align}
From this definition we obtain immediately that $D_n(x,t_X,t_Y)=-D_n(x,t_Y,t_X)$ and $D_n(x,t_X,t_Y)=0$ if $t_X=t_Y$.

\begin{lem}
	\label{lem:DInc}
	Assume that Condition \eqref{conditionI} holds, then $D_n(x,t_X,t_Y)$ is an increasing function of $t_X$, when $x,t_Y$ are kept fixed.
\end{lem}
\begin{proof}
	By induction. When $n=1$,
	\begin{equation}
	D_1(x,t_X,t_Y)=(t_X-t_Y)\int_{\mathbb R} \varphi(x+u)\bigl(f_1(u)-f_2(u)\bigr)\ud u
	\end{equation}
	is clearly an increasing function of $t_X$ when $x,t_Y$ are kept fixed.

	Now suppose Lemma \ref{lem:DInc} is proven for $n=k$. From Lemma \ref{lem:DeltaNtoN_1} we know that
	\begin{equation}
		\begin{aligned}
			&D_{k+1}(x,t_X,t_Y)\\
			=&\int_{\mathbb R} I_{\{t_X f_1(u)\ge t_Y f_2(u)\}}D_k(x+u,t_X f_1(u), t_Y f_2(u))\ud u\\
			+&\int_{\mathbb R} I_{\{t_X f_2(u)< t_Y f_1(u)\}}D_k(x+u,t_X f_2(u), t_Y f_1(u))\ud u.
		\end{aligned}
	\end{equation}
When $u,x,t_Y$ are fixed, $D_k(x+u,t_X f_1(u), t_Y f_2(u))$ and $D_k(x+u,t_X f_2(u), t_Y f_1(u))$ are increasing functions of $t_X$. And $D_k(x+u,t_X f_1(u), t_Y f_2(u))\ge 0$ when $t_X f_1(u)\ge t_Y f_2(u)$, $D_k(x+u,t_X f_2(u), t_Y f_1(u))\le 0$ when $t_X f_2(u)< t_Y f_1(u)$. So the two integrands in the above two integrals are both increasing functions of $t_X$.

	Hence we obtain that $D_{k+1}(x,t_X,t_Y)$ is an increasing function of $t_X$ when $x,t_Y$ are fixed, and complete the proof.
\end{proof}

Now let $t_X=\xi_0$, $t_Y=1-\xi_0$, then $D_n(x,t_X,t_Y)=\Delta_n(x,\xi_0)$. By Lemma \ref{lem:DInc} and $D_n(x,t_X,t_Y)=-D_n(x,t_Y,t_X)$, we get the desired fact.
\begin{cor}
	\label{cor:DeltaInc}
	For any fixed $x\in\mathbb{R}$ and $n\in\mathbb{Z}^+$, $\Delta_n(x,\xi_0)$ is an increasing function of $\xi_0$.
\end{cor}

Now we are ready to prove Theorem \ref{thm:main}.
\begin{proof}[Proof of Theorem \ref{thm:main}]
	Firstly, we assume that condition \eqref{conditionI} holds and prove the optimality of $\mathsf{M}^n$.
	This can be easily obtained by Theorem \ref{thm:R*} and Corollary \ref{cor:DeltaInc}. We now use mathematical induction.
	
	When $n=1$, by Corollary \ref{cor:DeltaInc},
	\begin{align}
	\Delta_{1}(x,\xi_0)=E_{\xi_0}\left[\varphi\left(x+X_{1}\right)\right]
	-E_{\xi_{0}}\left[\varphi\left(x+Y_{1}\right)\right]
	\end{align}
	is an increasing function of $\xi_0$ for any fixed $x$. Since $\Delta_{1}(x,\frac{1}{2})=0$, then the optimal strategy should chooses $X$ first if $\xi_0\ge \frac{1}{2}$, and chooses $Y$ first if $\xi<\frac{1}{2}$. This means that $\mathsf{M}^1$ is the optimal strategy of the $(\xi_0,1,x)$-bandit, for any $x\in\mathbb{R}$ and $\xi_0\in[0,1]$.
	
	Assume that for fixed $k\ge 1$, the myopic strategy $\mathsf{M}^k$ is optimal for $(\xi_0,k,x)$-bandits, for any $x\in\mathbb{R}$ and $\xi_0\in[0,1]$.

	Consider now a bandit problem with $k+1$ trials. Note that by \eqref{eq:Wcase}, there is
	\begin{align*}
		W(\xi_0,k+1,x,\mathsf{M}^{k+1})&=\begin{cases}
			W(\xi_0,k+1,x,\mathsf{L}^{k+1}),\quad\text{if $\xi_{0}\geq\frac{1}{2}$,}\\
			W(\xi_0,k+1,x,\mathsf{R}^{k+1}),\quad\text{otherwise.}
		\end{cases}
	\end{align*}
	And by \eqref{eq:LeqR_4} and \eqref{eq:LeqR_5}, we know that
	\begin{align}
		W(\xi_0,k+1,x,\mathsf{L}^{k+1})
		=&E_{\xi_0}\biggl[\varphi\biggl(x+\sum_{i=1}^{k+1} Z_i^{\mathsf{L}^{k+1}}\biggr)\biggr]\notag\\
		=& E_{\xi_{0}}\biggl[W\Bigl(\frac{\xi_0f_1(X_1)}{\xi_0f_1(X_1)+(1-\xi_0)f_2(X_1)},k,x+X_1,\mathsf{M}^{k}\Bigr)\biggr],\\
		W(\xi_0,k+1,x,\mathsf{R}^{k+1})
		=&E_{\xi_0}\biggl[\varphi\biggl(x+\sum_{i=1}^{k+1} Z_i^{\mathsf{R}^{k+1}}\biggr)\biggr]\notag\\
		=& E_{\xi_{0}}\biggl[W\Bigl(\frac{\xi_0f_2(Y_1)}{\xi_0f_2(Y_1)+(1-\xi_0)f_1(Y_1)},k,x+Y_1,\mathsf{M}^{k}\Bigr)\biggr].
	\end{align}

	By Corollary \ref{cor:DeltaInc}, $\Delta_{k+1}(x,\xi_0)$ is an increasing function of $\xi_0$, $\Delta_{k+1}(x,\frac{1}{2})=0$. And by the induction hypothesis, $\mathsf{M}^k$ is the optimal strategy of $(\frac{\xi_0 f_1(u) }{\xi_0 f_1(u) +(1-\xi_0)f_2(u)},k,x+u)$-bandits and $(\frac{\xi_0 f_2(u) }{\xi_0 f_2(u) +(1-\xi_0)f_1(u)},k,x+u)$-bandits, for $\forall u\in\mathbb{R}$. Then we have
	\begin{equation}
	\begin{aligned}
	&W(\xi_0,k+1,x,\mathsf{M}^{k+1})\\
		=&\max\left\{W(\xi_0,k+1,x,\mathsf{L}^{k+1}),W(\xi_0,k+1,x,\mathsf{R}^{k+1})\right\}\\
		=&\max\left\{E_{\xi_{0}}\biggl[V\biggl(\frac{\xi_0f_1(X_1)}{\xi_0f_1(X_1)+(1-\xi_0)f_2(X_1)},k,x+X_1\biggr)\biggr],\right.\\
		&\left.	\ \ \ \ \ \ \ \ \ \ \ \ \ \ \ E_{\xi_{0}}\biggl[V\biggl(\frac{\xi_0f_2(Y_1)}{\xi_0f_2(Y_1)+(1-\xi_0)f_1(Y_1)},k,x+Y_1\biggr)\biggr]\right\}.
	\end{aligned}
	\end{equation}
	Therefore, by Theorem \ref{thm:R*},
	\begin{align}
	W(\xi_0,k+1,x,\mathsf{M}^{k+1})=V(\xi_0,k+1,x).
	\end{align}
	Hence $\mathsf{M}^{k+1}$ is the optimal strategy of $(\xi_0,k+1,x)$-bandits, for any $x\in\mathbb{R}$ and $\xi_0\in[0,1]$. The first part of the main theorem is proved.
	
	Now we prove that, if the strategy $\mathsf{M}^n$ is optimal for $(\xi_0,n,x)$-bandits, for any integer $n\ge 1$, $\forall \xi_0\in[0,1]$, $\forall x\in \mathbb{R}$, then condition \eqref{conditionI} holds.
	
	Actually, we only need to consider the case when $n=1$. For any fixed $x\in\mathbb{R}$, we have
	\begin{align}
		E_{\xi_0}[\varphi(x+X_1)]&=\xi_0\int_{\mathbb R}\varphi(x+u)f_1(u)\ud u+(1-\xi_0)\int_{\mathbb R} \varphi(x+u)f_2(u)\ud u,\\
		E_{\xi_0}[\varphi(x+Y_1)]&=\xi_0\int_{\mathbb R}\varphi(x+u)f_2(u)\ud u+(1-\xi_0)\int_{\mathbb R} \varphi(x+u)f_1(u)\ud u.
	\end{align}
	
	Since the strategy $\mathsf{M}^1$ which chooses $X$ if and only if $\xi_0\ge \frac{1}{2}$ is the optimal strategy, we know that
	\begin{align}
	E_{\xi_0}[\varphi(x+X_1)]-E_{\xi_0}[\varphi(x+Y_1)]\ge 0,~\text{if }\xi_0\ge \frac{1}{2}.
	\end{align}
	Then for any fixed $x\in\mathbb{R}$, we have
	\begin{align}
		(2\xi_0-1)\Bigl[ \int_{\mathbb R}\varphi(x+u)f_1(u)\ud u-\int_{\mathbb R}\varphi(x+u)f_2(u)\ud u\Bigr]\ge 0,~\text{if }\xi_0\ge \frac{1}{2},
	\end{align}
	which leads to
	\begin{align}
		\int_{\mathbb R}\varphi(x+u)f_1(u)\ud u-\int_{\mathbb R}\varphi(x+u)f_2(u)\ud u\ge 0.
	\end{align}
	This is clearly condition (I).
\end{proof}

Theorem \ref{thm:main} gives a reasonable condition on $\varphi$ that is necessary and sufficient for myopic strategy $\mathsf{M}^n$ being the optimal strategy.
With Theorem \ref{thm:main} in hand, we can immediately obtain the following two corollaries. The first result applied Theorem \ref{thm:main} on the $\varphi(x)=x$ case, and it is the combination of the work in Feldman \cite{Feldman1962} and Kelley \cite{Kelley1974}. The second corollary answers Nouiehed and Ross's conjecture for two-armed case.
\begin{cor}
	Myopic strategy $\mathsf{M}^n$ is optimal for $(\xi_0,n,x)$-bandits with utility function $\varphi(x)=x$, for any integer $n\ge 1$, $\forall \xi_0\in[0,1]$, $\forall x\in \mathbb{R}$, if and only if
\begin{align}
	E_P[X_1\vert H_1]\ge E_P[Y_1\vert H_1].
\end{align}
\end{cor}

\begin{cor}[Nouiehed and Ross's conjecture]
	Consider the two-armed Bernoulli bandits. Let the distributions $F_1$ and $F_2$ be Bernoulli($\alpha$) and Bernoulli($\beta$), and $\alpha>\beta$. Under hypothesis $H_1$, experiment $X$ obeys $Bernoulli(\alpha)$, while $Y$ obeys $Bernoulli(\beta)$; under hypothesis $H_2$, $X$ obeys $Bernoulli(\beta)$, $Y$ obeys $Bernoulli(\alpha)$. Then for any fixed positive integers $k$ and $n$, Feldman's strategy $\mathsf{M}^n$ maximizes $P(S_n\ge k)=E_P[I_A]$, where $I_A$ is the indicator function of set $A=\{S_n\ge k\}$.
\end{cor}
\begin{proof}
	Note that the proof of Theorem \ref{thm:main} also holds in the case of the distributions being Bernoulli distributions. In this case, we only need to modify the calculation of expectations.
	
	For any fixed $k$, let the utility function be $\varphi(x)=I_{[k,+\infty)}(x)$. By Remark 2, this $\varphi$ satisfies condition \eqref{conditionI}, so Theorem \ref{thm:main} can be applied, and hence $\mathsf{M}^n$ maximizes the expectation $P(S_n\ge k)=E_P[I_A]=E_{\xi_0}[\varphi(S_n)]$.
\end{proof}

\section*{Acknowledgements}
\noindent We sincerely thank Professor Jordan Stoyanov for his valuable suggestions and amendments to this article.
%%%%%%%%%%%%%%%%%%%%%%%%%%%%%%%%%%%%%%%%%%%%%%
%% Funding information, if any,             %%
%% should be provided in the                %%
%% funding section.                         %%
%%%%%%%%%%%%%%%%%%%%%%%%%%%%%%%%%%%%%%%%%%%%%%
\section*{Funding information}
\noindent Chen gratefully acknowledges the support of the National Key R\&D Program of China (grant No. 2018YFA0703900), Shandong Provincial Natural Science Foundation, China (grant No. ZR2019ZD41) and Taishan Scholars Project.

%\bibliographystyle{plain} % Style BST file (imsart-number.bst or imsart-nameyear.bst)
%\bibliography{tab}       % Bibliography file (usually '*.bib')

\end{document}